\documentclass[a4paper]{article}

\usepackage[utf8]{inputenc}
\usepackage[colorlinks]{hyperref}
\usepackage{csquotes}
\usepackage{tikz}
\usepackage[margin=22mm]{geometry}

\usepackage{lmodern}
\usepackage{array}
\usepackage{graphicx}
\graphicspath{{./pic/}}
\usepackage{amssymb, amsfonts, amsmath, amsthm, mathrsfs}
\usepackage{enumerate}
\usepackage{enumitem}
\usepackage{epstopdf}
\usepackage[font=footnotesize]{caption}
\usepackage[colorlinks]{hyperref}
\usepackage{tikz}

\theoremstyle{plain}
\newtheorem{theorem}{Theorem}[section]

\newtheorem{lemma}[theorem]{Lemma}
\newtheorem{proposition}[theorem]{Proposition}

\newtheorem{thm}{Theorem}

\theoremstyle{definition}
\newtheorem{remark}[theorem]{Remark}

\numberwithin{equation}{section}

\newcommand{\N}{\mathbb{N}}

\newcommand{\R}{\mathbb{R}}

\newcommand{\calD}{\mathcal{D}}

\newcommand{\ind}[1]{\mathbf{1}_{\left\{#1\right\}}}

\newcommand{\crochet}[1]{{\big\langle #1 \big\rangle}}
\newcommand{\uppar}[1]{^{{\scriptscriptstyle (}#1{\scriptscriptstyle )}}}
\renewcommand{\bar}[1]{\overline{#1}}
\renewcommand{\tilde}[1]{\widetilde{#1}}
\renewcommand{\hat}[1]{\widehat{#1}}
\renewcommand{\phi}{\varphi}
\renewcommand{\epsilon}{\varepsilon}

\newcommand{\E}{\mathbf{E}}

\renewcommand{\P}{\mathbf{P}}
\DeclareMathOperator{\Var}{\mathbf{V}\mathrm{ar}}
\newcommand{\calF}{\mathcal{F}}

\newcommand{\dd}{\mathrm{d}}

\newcommand{\calN}{\mathcal{N}}
\newcommand{\Cov}{\mathbf{C}\mathrm{ov}}

\renewcommand{\rho}{\varrho}
\renewcommand{\epsilon}{\varepsilon}

\newcommand{\bbP}{\mathbb{P}}

\newcommand{\bbR}{\mathbb{R}}

\newcommand{\cE}{\mathcal E}
\newcommand{\cD}{\mathcal D}

\newcommand{\rmd}{{\rm d}}
\newcommand{\rme}{{\rm e}}

\newcommand{\diffd}{\mathrm{d}}

\newcommand{\frakD}{\mathfrak{D}}

\usepackage{xcolor}

\renewcommand{\epsilon}{\varepsilon}

\title{A simple backward construction of Branching Brownian motion with large displacement and applications}

\author{
Julien Berestycki\thanks{ \texttt{julien.berestycki@stats.ox.ac.uk} 
University of Oxford }
\and 
\'Eric Brunet\thanks{   \texttt{Eric.Brunet@lps.ens.fr}    Sorbonne
Universit\'e, Laboratoire de Physique Statistique, \'Ecole Normale Sup\'erieure, 
PSL Research University; Universit\'e Paris Diderot Sorbonne Paris-Cit\'e,
CNRS
}
\\  \and Aser Cortines\thanks{ \texttt{aser.cortinespeixoto@math.uzh.ch},
 Universit\"at Z\"urich,
Institute f\"ur Mathematik }
\and Bastien Mallein\thanks{\texttt{mallein@math.univ-paris13.fr}
LAGA, Universit\'e Paris 13}
}

\date{\today}

\begin{document}

\maketitle

\begin{abstract}
In this article, we study  the extremal processes of branching Brownian motions conditioned on having an unusually large maximum. The limiting point measures form a one-parameter family and are the decoration point measures in the extremal processes of several branching processes, including branching Brownian motions with variable speed and multitype branching Brownian motions. We give a new, alternative representation of these point measures and we show that they form a continuous family. 
This also yields a simple probabilistic expression for the constant that appears in the large deviation probability of having a large displacement. As an application, we show that Bovier and Hartung's \cite{BoH15} results about  variable speed branching Brownian motion also describe the extremal point process of branching Ornstein-Uhlenbeck processes.
\end{abstract}

\section{Introduction}

Spatial branching processes, and in particular, the behaviour of their extremal particles, have been at the centre of a wide research activity over the past few years, both in the physics \cite{BD09,BDMM06,DMS} and in the mathematical literature \cite{Aid,ABBS,ABK,Madaule}. These models have a rich and complex structure that is of intrinsic interest, but they are also representatives of an intriguing ``universality'' class, the so-called \emph{log-correlated fields} which includes the two-dimensional Gaussian free field \cite{BDZ16,BL18}, Gaussian multiplicative chaos \cite{rhodes2013gaussian}, random matrices \cite{ABB2017} and others.

Perhaps the simplest model in this class is the \textit{branching Brownian motion}, in which particles move in $\R$ as Brownian motions, branch into two particles at rate one and behave independently of each others. For the system started with a single particle at the origin, let $\mathcal{N}_t$ be the set of particles alive at time $t$ and for $u \in \mathcal{N}_t$ let $X_t(u)\in \R$ be its position. For $s\le t$ we will also write $X_s(u)$ for the position of the unique ancestor of $u$ at time $s$ so that $(X_s(u), s \leq t)$ is the path followed by the particle $u$. Then, it was proved in \cite{ABBS,ABK} that the point measure
\begin{equation}\label{BBM}
  \mathcal{E}_t := \sum_{u \in \calN_t} \delta_{X_t(u) - \sqrt{2} t + \frac{3}{2\sqrt{2}} \log t}
\end{equation}
converges in law, as $t \to \infty$ toward a \emph{random intensity decorated Poisson point process} (DPPP for short) $\mathcal{E}_\infty$.

In general, the law of a DPPP $\cE$ is characterized by a pair $(\nu, \frakD)$ where $\nu$ is a random sigma-finite measure on $\R$ and $\frakD$ is the law of a random point process on $\R$. The point measure $\cE$ can be constructed, conditionally on $\nu$, by first taking a realisation of a Poisson point process on $\R$ with intensity $\nu$, whose atoms are listed as $(x_i, i \in I)$, and an independent family of i.i.d.\@ point processes $(\cD_i, i \in I)$ with law $\frakD$. Then, each atom $x_i$ is replaced by the point process $\cD_i$, shifted by $x_i$ (this action is called the {\it decoration} of $x_i$ with a point process of law $\frakD$). In other words, writing $(d_i^j, j \in J_i)$ the atoms of the point process $\cD_i$, we have
\begin{equation}
  \label{eqn:defDPPP}
  \cE = \sum_{i \in I} \sum_{j \in J_i} \delta_{x_i + d^j_i}.
\end{equation}
We refer to \cite{SubZei} for an in-depth study of random intensity decorated Poisson point processes, and their occurrences as limit of extremal point measures.

With this notation, $\mathcal{E}_\infty := \lim_{t\to \infty} \cE_t$ is the following DPPP 
\begin{equation}\label{DPPP du BBM}
\mathcal{E}_\infty= \text{DPPP}( \kappa Z_\infty \rme^{-\sqrt{2} x} \diffd x, \frakD^1)
\end{equation}
where $\kappa$ is an implicit constant, $Z_\infty$ is the a.s.\@ positive limit of the so-called \emph{derivative martingale}
\begin{equation}
  Z_t := \sum_{u \in \mathcal{N}_t} \big(\sqrt{2} t - X_t(u)\big) \rme^{\sqrt{2}X_t(u) - 2 t},
\end{equation}
and where the \textit{decoration law} $\frakD^{1}$ is the law of a point measure supported on $(-\infty,0]$, with  an atom at $0$ defined by the following weak limit
\begin{equation}
 \label{eqn:defineDrho1}
   \mathfrak{D}^{1} (\cdot) := \lim_{t \to \infty} \textstyle \P \Big( \sum\limits_{u \in \mathcal{N}_t } \delta_{\{ X_t(u) - M_t \} } \in \cdot  \,  \big| \, M_t \geq \sqrt{2} t \Big),
\end{equation}
where $M_t := \max_{u \in \mathcal{N}_t} X_t(u)$. Moreover, it is well-known that $\max \cE_t$ converges in distribution toward $\max \cE_\infty$, where $\max \cE$ is the position of the largest atom in a point process $\cE$ (see Lalley and Selke \cite{LaS}).

The decoration law $\frakD^1$ belongs to the family $\left(\mathfrak{D}^\rho, \rho \in [1,\infty] \right)$, defined, for $\rho < \infty$ by the weak limits
\begin{equation}
 \label{eqn:defineDrho}
   \mathfrak{D}^{\rho} (\cdot) := \lim_{t \to \infty} \textstyle \P \Big( \sum\limits_{u \in \mathcal{N}_t } \delta_{\{ X_t(u) - M_t \} } \in \cdot  \,  \big| \, M_t \geq \sqrt{2} \rho t \Big).
\end{equation}
We denote by $\mathfrak{D}^\infty$ the law of the Dirac mass at $0$. The family $ \mathfrak{D}^{\rho} (\cdot)$ was introduced by Bovier and Hartung \cite{BoH15} as the decorations appearing in the extremal processes of variable speed branching Brownian motions.  A detailed statement of the result of Bovier and Hartung is given in Section \ref{sec:bou}. The decoration law $\mathfrak{D}^{\rho} (\cdot)$ can also appear in the context of multitype branching Brownian motions.

Note that the law $\mathfrak{D}^\rho$ is constructed by conditioning the branching Brownian motion on a large deviation event for its maximum. For  $\rho \in (1,\infty)$ we define
\begin{equation}
  \label{eqn:defineC}
 C(\rho) := \rho \lim_{t \to \infty} t^{1/2} \rme^{(\rho^2-1)t} \P(M_t \geq \sqrt{2} \rho t).
\end{equation}
The asymptotic behaviour of $\P(M_t > \sqrt 2 \rho t)$ was first studied in the seminal paper \cite{ChRo88} (where the existence of the limit $C(\rho)$ is implicit) and the function $C(\rho)$ plays a key role in \cite{BoH15} where it is proven that $C(1)=0$ and that $\lim_{\rho \to  \infty} C(\rho)=(4\pi)^{-1/2}$.
More recently, the same function $C(\rho)$ is the focus of \cite{DMS} where, in particular, the asymptotic behaviour of $C(\rho)$ as $\rho \to \infty$ and $\rho \to 1$ are conjectured. See \cite{DerridaShi,GantertHof,Burac} for further recent developments on this topic.

\begin{figure}[ht]
\centering
\input{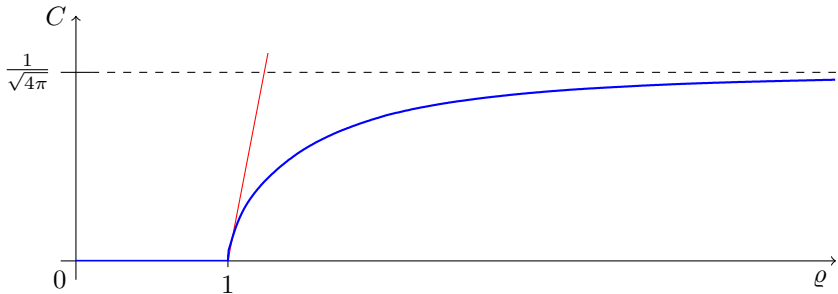}
\caption{An approximation of the function $C$, computed using its representation from Theorem~\ref{prop_continuity_decoration_fctn}, together with its right derivative at $\rho = 1$.}
\end{figure}

The goal of this article is to study both the function $\rho \mapsto C(\rho)$ and the family $(\mathfrak{D}^\rho, \rho \in (1,\infty])$. We provide a new construction of these quantities, that do not rely on the conditioning on a vanishing event but uses a \emph{spine decomposition}. Recall that a sequence of random point measures $(\mathcal{P}_t)_{t \geq 0}$ on $\bbR$ converges to $\mathcal{P}$ in law for the topology of vague convergence if and only if, for every compactly supported continuous function $\phi$, the real valued random variables
\begin{equation}
  \label{eqn:defineCrochet}
  \crochet{\mathcal{P}_t,\phi} := \int \phi(x) \mathcal{P}_t (\rmd x )
\end{equation}
converge in law to $\crochet{\mathcal{P},\phi}$ as $t\to \infty$. We prove in this article that $C$ is continuous on $[1,\infty]$, and that $\rho \mapsto \mathfrak{D}^\rho$ is continuous on $(1,\infty]$ for the topology of vague convergence. This can be used to extend the main theorem of \cite{BoH15}.

Let $(B_t, t \geq 0)$ be a standard Brownian motion, $(\sigma_k, k \in \N)$ be the ranked atoms of a Poisson point process with intensity $2 \, \rmd x$ on $\R_+$ and $(X^{(k)}_t(u), u \in \calN^{(k)}_t, t \geq 0)$ for $ k \in \N$ be i.i.d.\@ branching Brownian motions. We shall assume that $B$, $(\sigma_k, k \geq 1)$ and $(X^{(k)}, k \geq 1)$ are independent of one another. Given $\rho \in (1,\infty)$ and $t \geq 0$, we define the point process
\begin{equation}
  \label{eqn:tildecalD}
  \tilde{\mathcal{D}}^\rho = \delta_0 + \sum_{k \in \N } \sum_{u \in \calN^{(k)}_{\sigma_k}} \delta_{B_{\sigma_k} - \sqrt{2} \rho \sigma_k + X^{(k)}_{\sigma_k}(u) }.
\end{equation}
In words, $\tilde{\mathcal{D}}^\rho$ is the point process constructed using a Brownian motion with drift $-\sqrt{2}\rho$, that spawns branching Brownian motions at rate $2$. A branching Brownian motion spawned at time $\sigma_k$ then starts evolving backward in time until it hits time $0$, the particles alive at that time are added to the point process.

\begin{figure}[ht]
\begin{center}
\includegraphics[width=5cm]{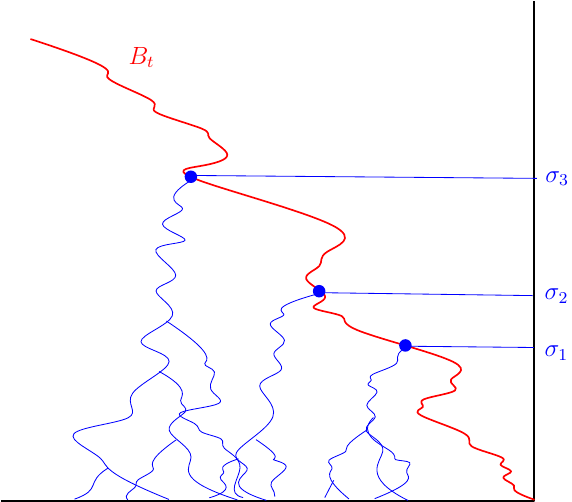}
\end{center}
\caption{Construction of the point process $\tilde{\mathcal D}^\rho$.}
\end{figure}

\begin{theorem} \label{prop_continuity_decoration_fctn}
Let $C : [1,\infty] \mapsto \R_+$ be the function given by \eqref{eqn:defineC} and for $\rho \geq 1$ let $\mathcal D^\rho$ be a random point measure of law  $\frakD^\rho$ as defined in \eqref{eqn:defineDrho}. Then
\begin{enumerate}[label=(\roman*)]
  \item $C(\rho) = \frac{1}{\sqrt{4\pi}}\P\big(\tilde{\mathcal{D}}^\rho((0,\infty)) = 0\big)$ for all $\rho>1$. The function $C$ is continuous on $[1,\infty]$. It also satisfies ${C}(1)=0$, $C(\rho)>0$ for $\rho>1$ and ${C}(\infty) = 1/\sqrt{4\pi}$. 
 \label{lem:continuityTildeC}
 \item $\P \big( \mathcal{D}^\rho \in \cdot \big) = \P \big( \tilde{\cD}^\rho \in \cdot \mid \tilde{\cD}^\rho((0,\infty)) = 0 \big) $. The family of point processes $\left(\mathcal{D}^\rho, \, \rho \in (1, \infty] \right)$ is continuous in the space of Radon point measures equipped with the topology of vague convergence.
 \label{lem:continuitycalD}
\end{enumerate}
\end{theorem}

The rest of the article is organized as follows. In Section \ref{sec:spine} we introduce the \emph{spinal decomposition} of the branching Brownian motion, and its application to the extremal process of the branching Brownian motion, seen from the rightmost particle. We then prove Theorem \ref{prop_continuity_decoration_fctn} in Section \ref{sec:pf}. Then, as an application of Theorem \ref{prop_continuity_decoration_fctn}, in Section \ref{sec:bou} we  show how the results of Bovier and Hartung \cite{BoH14,BoH15} about variable speed branching Brownian motion also describe the extremal point process of  a branching Ornstein-Uhlenbeck. We conclude this article with some open questions.

\section{Spinal decomposition at the maximum}
\label{sec:spine}

We apply the so-called \emph{spinal decomposition} of the branching Brownian motion to obtain the joint law of the maximum and   the extremal process of the branching Brownian motion. The \emph{spinal decomposition} is an alternative description of the process constructed via the \emph{probability tilting} by the additive martingale $W^{\sqrt{2}\rho}$, which is defined for all $t \geq 0$ by
\begin{equation}
  \label{eqn:defMartingale}
  W_t^{\sqrt{2} \rho} := \sum_{u \in \mathcal{N}_t} \mathrm{e}^{\sqrt{2} \rho X_t(u) - (\rho^2 + 1) t}.
\end{equation}
This idea was pioneered by Lyons, Peamantle and Peres in \cite{LPP95} to study Galton-Watson processes, then generalized to branching random walks by Lyons \cite{Lyo97} and to general branching processes in \cite{BiK04}. 

Let $(\calF_t)$ be the natural filtration of the branching Brownian motion, defined by
\[
  \calF_t = \sigma\left( \calN_s, (X_s(u), u \in \calN_s), s \leq t \right).
\]
For $\rho \in \R$ and $t \geq 0$, we introduce the size-biased law as 
\begin{equation}
  \label{tilting}
  \bar{\P}_\rho\big|_{\calF_t} = W^{\sqrt{2}\rho}_t \cdot \P \big|_{\calF_t},
\end{equation} 
and call $X$ under $\bar{\P}_\rho$ the size biased process. 

The spinal decomposition links the size biased process with the so-called \emph{branching Brownian motion with spine}. It describes the evolution of a branching particle system with a distinguished particle $\xi_t$, which behaves differently from the others. The system starts with the spine particle at position $0$. This particle moves according to a Brownian motion with drift $\sqrt{2} \rho$ and produces children at rate $2$. Each of its children starts an independent (standard) branching Brownian motion from its birth place. We shall use the same notation $\calN_t$ for the set of particles alive at time $t$ in this process (it is not a Yule process anymore), and write $\xi_t \in \calN_t$ for the label of the spine particle.
The law of this branching Brownian motion with spine is denoted by $\hat{\P}_\rho$. The spinal decomposition can be stated as follows.

\begin{thm}[Spinal decomposition \cite{Bere}]
\label{them_spinal_dec}
For all $\rho \in \R$, with the above notation we have
$ \bar{\P}_\rho\big|_{\calF_t} =\hat{\P}_\rho \big|_{\calF_t}$ for all $t \geq 0$. Moreover, for all $u \in \mathcal{N}_t$,
\[\hat{\P}_\rho\left( \xi_t = u \mid \calF_t \right) = \frac{\rme^{\sqrt{2} \rho X_t(u) - t (\rho^2 + 1)}}{W_t^{\sqrt 2\rho}}.\]
\end{thm}

In words: the law of the marked tree $((X_s(u), u \in \calN_s), s\leq t)$ has same law under probability $\hat{\P}$ and $\bar{\P}$. Moreover, conditionally on this marked tree, one can choose to distinguish at random an individual $u \in \calN_t$ with probability proportional to $\rme^{\sqrt{2} \rho X_t(u)}$ to construct the law of the branching Brownian motion with spine.

Using this result, we can describe the joint law of the extremal process and the maximal displacement of the branching Brownian motion.
\begin{lemma}
\label{lem:critique}
Let $\rho \geq 1$ and $t \geq 0$, we denote by
\[
  \textstyle{\cE}^*_t = \sum_{u \in \mathcal{N}_t} \delta_{X_t(u) - M_t}
\]
the extremal process of the branching Brownian motion seen from the rightmost individual, and we introduce the point process
\begin{align}
  \label{eqn:tildecalDt}
  \tilde{\mathcal{D}}^\rho_t & = \delta_0 + \sum_{k \in \N : \sigma_k \leq t} \sum_{u \in \calN^{(k)}_{\sigma_k}} \delta_{B_{\sigma_k} - \sqrt{2} \rho \sigma_k+ X^{(k)}_{\sigma_k}(u) },
\end{align}
where $B$ is a Brownian motion, $(\sigma_k, k \geq 1)$ are the jump times of a Poisson process with intensity $2$ and $(X^{(k)}(u)_s, u \in \mathcal{N}_s, s \geq 0)_{k \in \N}$ are i.i.d. branching Brownian motions. For all non-negative measurable function $f,F$, we have
\[
  \E\left( F({\cE}^*_t) f(M_t-\sqrt{2} \rho t) \right) = \rme^{(1-\rho^2) t}\E\left( \rme^{\sqrt{2}\rho B_t} f(-B_t) F\left( \tilde{\calD}^\rho_t\right) \ind{ \tilde{\calD}^\rho_t((0,\infty))=0} \right).
\]
\end{lemma}

\begin{proof}
For $t \geq 0$, denote by $u_t^{\mathrm{tip}} \in \calN_t$ the label of the largest particle alive at time $t$ (which is a.s.\@ unique). We observe that we can write
\[
  \E\left( F({\cE}^*_t)f(M_t - \sqrt{2}\rho t)  \right) = \E\left(\sum_{u \in \calN_t} F\big({\cE}^*_t(u) \big) \ind{u = u_t^{\mathrm{tip}}} f(M_t - \sqrt{2}\rho t) \right),
\]
where ${\cE}^*_t(u) := \sum_{v \in \calN_t} \delta_{X_t(v) - X_t(u)}$ is the extremal point measure seen from particle $u \in \calN_t$. Thanks to the spinal decomposition and using \eqref{tilting}, the above reads 
\begin{align*}
 \E\left( F({\cE}^*_t)f(M_t - \sqrt{2}\rho t) \right) &= \bar{\E}_\rho  \left( \frac{1}{W_t^{\sqrt 2\rho}}\sum_{u \in \calN_t}  F\big({\cE}^*_t(u) \big) \ind{u=u_t^{\mathrm{tip}}} f(M_t - \sqrt{2} \rho t) \right)\\
  &= \hat{\E}_\rho \left( \rme^{-\sqrt{2} \rho X_t(\xi_t) + (\rho^2 + 1)t} F({\cE}^*_t(\xi_t)) \ind{\xi_t=u_t^{\mathrm{tip}}} f(X_t(\xi_t) - \sqrt{2} \rho t) \right).
\end{align*}

Next, we use the definition of the branching Brownian motion with spine to rewrite the above expression. For $s\in [0,t]$ let $B_s = Z_{t-s}-Z_t$ where  $Z_s = X_t(\xi_s) - \sqrt{2} \rho s$ and for all $k \in \N$, $\sigma_k$ is the $k$th instant at which the spine gives birth to a new particle when running time backward from $t$ (i.e.\@ $t-\sigma_1$ is the last time before $t$ at which the spine branches).Then, under $\hat \E_\rho$, $B$ is a standard Brownian motion and $(\sigma_k)$ are the atoms of a Poisson point process on $\R_+$ with intensity measure~$2 \rmd x$. For each branching event $\sigma_k$, the spine gives birth to a standard branching Brownian motion that we call $X^{(k)} \equiv (X^{(k)}_s(u), u \in \calN^{(k)}_s; \, s \in \bbR_+)$. With these notation we then get 
\begin{equation} \label{Point_process_seen_from_the_tip_BM}
  {\cE}^*_t(\xi_t) =\sum_{k \in \N : \sigma_k \leq t} \sum_{u \in \calN^{(k)}_{\sigma_k}} \delta_{B_{\sigma_k} - \sqrt{2} \rho \sigma_k + X^{(k)}_{\sigma_k}(u)}.
\end{equation}

All that is left to do is thus to note that  under $\hat \E_\rho$, the pair of variables $({\cE}^*_t(\xi_t), X_t(\xi_t) ) $ jointly have the same law as  $(\tilde{\calD}^\rho_t, (-B_t+\sqrt 2 \rho t))$ from \eqref{eqn:tildecalDt}. Thus substituting 
 $ \ind{\xi_t = u_t^{\mathrm{tip}}}$ by $ \ind{\tilde{\calD}^\rho_t((0,\infty))=0}$ and  $ f(X_t(\xi_t) - \sqrt{2} \rho t)$ by $f(-B_t) $ we conclude that
\[
  \E\left( F({\cE}^*_t)f(M_t - \sqrt{2}\rho t)  \right) = \rme^{(1-\rho^2)t}\E\left( \rme^{\sqrt{2} \rho B_t} f(-B_t) F(\tilde{\calD}^\rho_t) \ind{\tilde{\calD}^\rho_t((0,\infty))=0} \right). \qedhere
\]
\end{proof}

We are now going to show that for all $\rho > 1$, the point measure $\tilde{\mathcal{D}}^\rho$ given in \eqref{eqn:tildecalD} is well-defined as the increasing limit of $\tilde{\calD}^\rho_t$ as $t \to \infty$. Recall that, with the notation of Lemma~\ref{lem:critique}, we have
\begin{align}
\label{eqn:definitiondelamesureDhrotilde}
 & \tilde{\mathcal{D}}^\rho = \delta_0 + \sum_{k \in \N } \sum_{u \in \calN^{(k)}_{\sigma_k}} \delta_{B_{\sigma_k}  - \sqrt{2} \rho \sigma_k+ X^{(k)}_{\sigma_k}(u)}.
\end{align}
Our first step is to prove that $\tilde{\mathcal{D}}^\rho$ is a well-defined sigma-finite point measure, meaning that for every  $a < b \in (-\infty, \infty]$, we have $\tilde{\calD}^\rho \big( [a, b]\big) <\infty$ a.s.\@.

\begin{lemma}
\label{lem:wellDefinition}
For all $\rho > 1$, $\tilde{\calD}^\rho$ is a well-defined point measure. Moreover, we have
\[
  \lim_{t \to \infty} \tilde{\calD}^\rho_t = \tilde{\calD}^\rho \quad \text{ a.s.\@ for the topology of the vague convergence.}
\]
\end{lemma}

\begin{remark}
Note that the above result would not hold for $\rho = 1$, as in that case one can prove that $\lim_{t \to \infty} \tilde{\calD}^1_t(0,1) = \infty$ a.s.
\end{remark}

\begin{proof}
Let $\rho > 1$, the point measure $\tilde{\calD}^\rho $ can be rewritten as
\begin{equation}\label{eq:reecriture_cluster_law}
  \tilde{\calD}^\rho = \delta_0 + \sum_{k \in \N } \sum_{u \in \calN^{(k)}_{\sigma_k}} \delta_{\left(B_{\sigma_k} - \sqrt{2}(\rho - 1) \sigma_k\right) + \left(X^{(k)}_{\sigma_k}(u) - \sqrt{2} \sigma_k\right)}.
\end{equation}
We observe that $\left(B_{\sigma_k} - \sqrt{2}(\rho - 1) \sigma_k, k \geq 0\right)$ is a random walk with negative drift $-\sqrt{2}(\rho-1)/2$. Moreover, for all $k \in \N$ the position of the largest atom in the point measure $\sum_{u \in \calN^{(k)}_{\sigma_k}} \delta_{X^{(k)}_{\sigma_k}(u) - \sqrt{2} \sigma_k}$  is, for large values of $k$, typically around position $-\frac{3}{2\sqrt{2}} \log \sigma_k \approx - \frac{3}{2\sqrt{2}} \log k$. Thus, heuristically, if $\rho > 1$, the random walk drifts to $-\infty$ at positive speed such that only a finite number of branching Brownian motions put particles in any given compact set. On the other hand,  when $\rho = 1$, the random walk $B_{\sigma_k}$ has  drift zero and we show that it implies that an infinite number of particles are to be found in any finite neighbourhood of $0$.

To make the above argument rigorous, we write $M_t$ for the maximal displacement at time $t$ in a branching Brownian motion. Setting $m_t = \sqrt{2} t - \frac{3}{2\sqrt{2}} \log t$, It is well-known \cite{ABBS,ABK} that $(M_t - m_t, t \geq 0)$ is tight and has uniform exponential tails. More precisely, it is proved in \cite{fang2012} (in a much more general settings) there exists $C>0$ and $\lambda > 0$ such that 
\begin{equation}
 \label{eqn:estimateTail}
 \P\left( |M_t - m_t|  \geq  x\right) \leq C \rme^{-\lambda x} \text{ for all $t, x >0$}.
\end{equation}

Given $k \in \N$, we denote by $M^{(k)} = \max_{u \in \calN^{(k)}_{\sigma_k}} X^{(k)}_{\sigma_k}(u) $ the maximal displacement of $X^{(k)}$ at time $\sigma_k$. Using the bounds from \eqref{eqn:estimateTail}, we observe immediately, using the Borel-Cantelli Lemma and the fact that $\sigma_k \sim_{k \to \infty} k/2$ that, with probability one,
\begin{equation} \label{eq_as_convergence_BBM_devided_log}
  \limsup_{k \to \infty} \frac{\left|M^{(k)} - \sqrt{2} \sigma_k\right|}{\log k} \leq \frac{3}{2\sqrt{2}} + \lambda^{-1}.
\end{equation}

In view of \eqref{eq_as_convergence_BBM_devided_log} and the law of large numbers, we deduce that
\begin{equation}\label{eq:bounds_decorated_rw}
  \lim_{k \to \infty} \frac{1}{k}\left( B_{\sigma_k} + M^{(k)} - \sqrt{2} \rho \sigma_k \right) = - \frac{\sqrt{2} (\rho-1)}{2}< 0 \quad \text{a.s.}
\end{equation}
In particular, it implies that given $A > 0$ one can find a random $T \in \bbR_+$ such that
\[ \forall \sigma_k \geq T, \quad B_{\sigma_k} + M^{(k)} - \sqrt{2} \rho \sigma_k \leq - A,\] in which case $\tilde{\calD}^\rho((x,\infty)) = \tilde{\calD}^\rho_{t}((x,\infty))$ for all $t > T$ and $x>-A$. This proves that $\tilde{\calD}^\rho$ is locally finite a.s.\@ and that $\tilde{\calD}^\rho_t\nearrow  \tilde{\calD}^\rho$ as $t \to \infty$, as claimed.
\end{proof}

Next we show the weak continuity of the family $(\tilde{\calD}^\rho, \varrho >1)$. 
\begin{lemma}
\label{lem:wellDefinition_continuity}
The family of point processes $(\tilde{\calD}^\rho, \varrho >1)$ is a.s.\@ continuous in $\rho > 1$. Moreover, 
for all $\rho > 1$, $\P(\tilde{\calD}^\rho((0,\infty)=0)>0$ and
\[
  \lim_{t \to \infty} \P\left( \tilde{\calD}^1_t \big((0,\infty) \big) = 0 \right) = 0 \quad \text{and} \quad  \lim_{\rho\to1}\P\left( \tilde{\calD}^\rho \big((0,\infty) \big) = 0\right)= 0.
\]
\end{lemma}

\begin{proof}
To prove the a.s.\@ continuity of $(\tilde{\calD}^\rho, \rho > 1)$, it is enough to show that for all continuous function $\phi$ with compact support, the function $\rho \mapsto \crochet{\phi,\tilde{\calD}^\rho}$ is continuous a.s.\@ This is a direct consequence of the fact that there are only finitely many atoms in any compact interval, and that the position of these atoms in $\calD^\rho$ are decreasing and continuous with $\rho$, by \eqref{eqn:tildecalD}. Hence, for any $\rho_0>1$, there is only a finite number of atoms to follow as $\rho$ increases to compute $\rho \in [\rho_0,\infty) \mapsto \crochet{\phi,\tilde{\calD}^\rho}$. Hence this function is continuous, which completes the proof of the first statement. For the second statement,  it suffices to observe that for $T>0$ there is positive probability that $\sigma_1>T$ and that $\tilde{\calD}^\rho(0,\infty)-\tilde{\calD}^\rho_{T}(0,\infty)=0.$

We now focus on the case  $\rho = 1$.  By law of iterated logarithms for the random walk,  we have that
\[ \limsup_{k \to \infty} k^{-1/2} B_{\sigma_k} = \infty \quad \text{ a.s.},\] 
which together with \eqref{eq_as_convergence_BBM_devided_log} yields 
\[
\limsup_{k \to \infty} \frac{B_{\sigma_k} + M^{(k)} - \sqrt{2}\sigma_k}{k^{1/2}} = \infty
\quad \text{a.s.}
\] 
This shows that the event $\{B_{\sigma_k} + M^{(k)} - \sqrt{2}\sigma_k \geq a \text{ infinitely often}\}$ has probability $1$ for every $a>0$. In particular it implies that $\tilde{\calD}^{1}_t((a,\infty)) \uparrow \infty$ a.s.\@ as $t \to \infty$. 

To conclude the proof, we observe that for all $\epsilon > 0$, there exists $t > 0$ such that
$ \P(\tilde{\calD}^{1}_t((0,\infty)) = 0) < \epsilon$. 
At the same time it follows from \eqref{eqn:tildecalDt} that $\tilde{\calD}^\rho_t$ is continuous in $\rho \in \R$, hence for all $\rho > 1$ small enough, we have 
\[
  \P\left(\tilde{\calD}^\rho((0,\infty)) = 0\right) \leq \P\left( \tilde{\calD}^\rho_t((0,\infty)) = 0 \right) \leq 2\epsilon,
\]
which shows that $\lim_{\rho \to 1} \P\left(\tilde{\calD}^\rho((0,\infty)) = 0\right)  = 0$, completing the proof.
\end{proof}

\section{Probabilistic representation of the extremal point process conditioned on a large maximum}
\label{sec:pf}

In this section, we prove the weak continuity in $\varrho$ of the cluster point process $\mathfrak{D}^\rho$ as well as the continuity of the function $\varrho \mapsto C(\rho)$ and their spine representation. To prove this, we show that the cluster law $\frakD^\rho$ and the function $C$ can be computed as continuous functionals of $\tilde{\frakD}^\rho$ defined in \eqref{eqn:tildecalD}. This connection is based an application of Lemma~\ref{lem:critique} to the study of the extremal process of the branching Brownian motion conditioned on having a maximum larger than $\sqrt{2} \rho$. Those results in combination complete the proof of Theorem~\ref{prop_continuity_decoration_fctn}.

We begin with the following computation of the extremal process of the branching Brownian motion conditioned on satisfying $\{M_t \geq \sqrt{2} \rho t\}$.
\begin{lemma}
\label{lem:computation}
Let $\rho>1$  and $\varphi \,: \bbR \mapsto \bbR_+$ be a continuous
function whose support is bounded from the left. Then
\begin{equation*}
  \lim_{t \to \infty} \rho t^{1/2} \rme^{(\rho^2-1)t} \E\left( \rme^{-\crochet{{\cE}^*_t,\phi}} \ind{M_t \geq \sqrt{2}\rho t} \right)
  = \frac{1}{\sqrt{4\pi}}\E\left( \rme^{-\crochet{\tilde{\calD}^\rho,\phi}} \ind{\tilde{\calD}^\rho((0,\infty)) = 0} \right). 
\end{equation*}
\end{lemma}

\begin{proof}
Fix $\phi$ as in the lemma and $\rho>1$.
Using Lemma~\ref{lem:critique}, we may write for $t > 0$
\begin{align}
  \rme^{(\rho^2-1)t} \E\left( \rme^{-\crochet{{\cE}^*_t,\phi}} \ind{M_t
\geq \sqrt{2}\rho t} \right) = &\E\Big( \rme^{\sqrt{2} \rho
B_t} \ind{B_t \leq 0}\rme^{-\crochet{\tilde{\calD}^\rho_t,\phi}}
\ind{\tilde{\calD}^\rho_t((0,\infty))=0}\Big).\nonumber
\end{align}
We compute the right hand side by first conditioning on $B_t=x$.
Introducing the point measure $\tilde{\calD}^{\rho,x}_t$ as
$\tilde{\calD}^\rho_t$ conditioned on $\{B_t=x\}$, one gets
\begin{align} 
  \rme^{(\rho^2-1)t} \E\left( \rme^{-\crochet{{\cE}^*_t,\phi}} \ind{M_t
\geq \sqrt{2}\rho t} \right) =
\int_{-\infty}^0\frac{\diffd x}{\sqrt{2\pi t}} \rme^{\sqrt2\rho
x-\frac{x^2}{2t}} \E\Big( \rme^{-\crochet{\tilde{\calD}^{\rho,x}_t,\phi}}
\ind{\tilde{\calD}^{\rho,x}_t((0,\infty))=0}\Big).\nonumber
\end{align}
We are going to show that, for any fixed $x<0$,
\begin{equation}
\lim_{t\to\infty}\E\Big( \rme^{-\crochet{\tilde{\calD}^{\rho,x}_t,\phi}}
\ind{\tilde{\calD}^{\rho,x}_t((0,\infty))=0}\Big)
=\E\Big( \rme^{-\crochet{\tilde{\calD}^{\rho},\phi}}
\ind{\tilde{\calD}^{\rho}((0,\infty))=0}\Big)
\label{limrhox}
\end{equation}
then, the lemma follows by a simple application of the dominated
convergence Theorem.

We shall couple the processes
$\tilde{\calD}^{\rho,x}_t$ and $\tilde{\calD}^{\rho}$ in such a way, that
for any fixed $x\in\R$,
\begin{equation}
 \lim_{t\to\infty} \crochet{\tilde{\calD}^{\rho,x}_t,\phi}
= \crochet{\tilde{\calD}^{\rho},\phi}\quad\text{a.s.}
\label{ascvg}
\end{equation}
Then, this gives \eqref{limrhox} (and the lemma) by dominated convergence.


Fix $x\in\R$, recall that $B$ is the Brownian underlying the
construction of $\tilde{\calD}^{\rho}$ and introduce  for $0\le s\le t$
$$\beta^{(t)}_s := B_s + \frac{s}{t}(x-B_t).$$
It is well-known that $(\beta^{(t)}_s; \, s \in [0,t])$ is a Brownian
bridge from $\beta^{(t)}_0=0$ to $\beta^{(t)}_t=x$.

Almost surely, there exists a random constant $C$ such that
$$|B_s| \le 1+ C s^{0.51}\quad\text{for all $s\ge0$}.$$
Then, with the same constant $C$, one checks that we have the following
uniform bound:
\begin{equation}
|\beta^{(t)}_s| \le 2 + |x| + C s^{0.51}  +s Ct^{-0.49} \quad\text{for all $t\ge0$ and all
$s\in[0,t]$}.
\label{defbeta}
\end{equation}

Recall that $\phi$ has bounded support on the left. Therefore, there exists $a \in \R$ such that $\phi(x)=0$ for all $x<a$. Let us fix $0<\epsilon<\sqrt 2(\rho-1)$. For all $t$ large enough so that $C t^{-0.49}<\epsilon$, observe that 
\[
\beta^{(t)}_s -\sqrt 2 \rho s \le -(\sqrt 2 \rho -\epsilon) s +2+|x| +Cs^{0.51} \quad\text{for  all
$s\in[0,t]$}.
\]
As in the proof of Lemma~\ref{lem:wellDefinition}, since $\sqrt 2 \rho -\epsilon>\sqrt 2$, we conclude   that there exists $T'<\infty$ a.s.\@ such that \emph{uniformly in
$t$}, all the
points in $\tilde{\calD}^{\rho,x}_t$ on the right of $a$ come from
branching events on the spine that occurred at times $\sigma_k\le T'$. 


Therefore, in computing $\crochet{\tilde{\calD}^{\rho,x}_t,\phi}$, one
only needs to consider finitely many points: those that branched from the
spine at a time smaller than $T'$. These points converge, as $t\to\infty$
to the corresponding points in $\tilde{\calD}^{\rho}$ (because
$\beta_s\uppar t\to B_s$ as $t\to\infty$) and, as $\phi$ is
continuous, \eqref{ascvg} holds and the lemma is proved.
\end{proof}

Using that last result, we now prove Theorem~\ref{prop_continuity_decoration_fctn}.
\begin{proof}[Proof of  Theorem~\ref{prop_continuity_decoration_fctn}]
We recall from \eqref{eqn:defineC} that for all $\rho > 1$, we have
\[
  C(\rho) = \rho\lim_{t \to \infty} t^{1/2}\rme^{(\rho^2-1)t} \P(M_t \geq \sqrt{2} \rho t).
\]
Therefore applying Lemma~\ref{lem:computation} with $\phi\equiv 0$, we can rewrite ${C}(\rho)$ as
\begin{equation}
  \label{eqn:alternativeFormula}
  {C}(\rho) = \frac{1}{\sqrt{4\pi}} \P\left( \tilde{\calD}^\rho((0,\infty))=0 \right).
\end{equation}
We deduce from Lemma~\ref{lem:wellDefinition_continuity} that ${C}$ is a continuous function on $[1,\infty)$ such that ${C}(1) = 0$. Additionally, it can be seen from the proof of \cite[Lemma 3.3]{BoH15} that $\lim_{\rho \to \infty} C(\rho) = \frac{1}{\sqrt{4\pi}} =: C(\infty)$, which completes the proof of the first part of Theorem~\ref{prop_continuity_decoration_fctn}.

We now turn to the proof of the second part. We recall that by the definition \eqref{eqn:defineDrho}, given $\mathcal{D}^\rho$ a point process of law $\frakD^\rho$, for all continuous function $\phi$ with compact support, we have
\[
  \E\left( \rme^{-\crochet{\calD^\rho,\phi}} \right) = \lim_{t \to \infty} \E\left( \rme^{-\crochet{{\mathcal E}^*_t,\phi}} \middle| M_t \geq \sqrt{2} \rho t \right).
\]
At the same time, by Lemma~\ref{lem:computation} we get
\begin{equation*}
  \lim_{t \to \infty} \E\left( \rme^{-\crochet{{\mathcal E}^*_t,\phi}} \middle| M_t \geq \sqrt{2} \rho t \right)
  = \lim_{t \to \infty} \frac{ \E\left( \rme^{-\crochet{{\mathcal E}^*_t,\phi}} \ind{M_t \geq \sqrt{2} \rho t} \right)}{\P(M_t \geq \sqrt{2} \rho t)}
  = \frac{\E\left(\rme^{-\crochet{\tilde{\calD}^\rho,\phi}}  \ind{\tilde{\calD}^\rho((0,\infty)) = 0} \right)}{\P(\tilde{\calD}^\rho((0,\infty)) = 0)}.
\end{equation*}
This shows that for all continuous compactly supported function $\phi\, : \bbR \mapsto \bbR_+$
\[
\E\left( \rme^{-\crochet{\calD^\rho,\phi}} \right) = \E\left( \rme^{-\crochet{\tilde{\calD}^\rho,\phi}} \middle| \tilde{\calD}^\rho((0,\infty)) = 0 \right),
\]
proving that $\bbP (\calD^\rho \in \cdot ) = \bbP \big( \tilde{\calD}^\rho \in \cdot \mid \tilde{\calD}^\rho((0,\infty)) = 0\big)$. The weak continuity of $\frakD^\rho$ for $\rho \in(1,\infty)$ follows readily from Lemma~\ref{lem:wellDefinition} and the continuity of $C$. This concludes the proof.
\end{proof}

\subsection{An alternative proof for the first part of Theorem~\ref{prop_continuity_decoration_fctn}}
We sketch here an alternative proof for the representation of $C(\rho)$ in terms of the point processes $\tilde {\mathcal D}^\rho$ defined in \eqref{eqn:tildecalD}. This proof is based on PDE analysis rather than tight probabilistic estimates, and can thus be of independent interest.

Let $M_t$ be the maximum at time $t$ in a branching Brownian motion, and set $u(x,t)=\mathbf P(M_t>x)$ its tail distribution. We recall that $u$ is solution to the Fisher-KPP equation $\partial_t u=\frac12\partial_x^2u +u-u^2$ with step initial condition $u(x,0)=\mathbf{1}_{\{x<0\}}$. We can thus compute $C(\rho)$ from its definition \eqref{eqn:defineC} using the Feynman-Kac representation to evaluate $\mathbf P(M_t>\sqrt2\,\rho t)=u(\sqrt2\,\rho t,t)$.

Recall from Feynman-Kac that, given a function $K(x,t)$, the solution to $\partial_t h=\frac12\partial_x^2 h+K h$ can be written as
\[h(x,t)=\E^x\left[h(B_t,0)\exp\left(\int_0^t\diffd s\,K(B_s,t-s)\right)\right].\]
We apply this not to $u(x,t)$, but to $\partial_xu(x,t)$, the \emph{derivative} of the solution to the Fisher-KPP equation, which is solution to $\partial_t[\partial_x u] =\frac12\partial_x^2[\partial_x u]+(1-2u)[\partial_x u]$ with initial condition $\partial_x u(x,0)=-\delta(x)$. This gives
\begin{align*}
\partial_xu(x,t)&=-\text \rme^t \E^x\Big[\delta(B_t) \text \rme^{-2\int_0^t\diffd s\,u(B_s,t-s)}\Big]\\
                &=-\frac1{\sqrt{2\pi t}}\text \rme^{t-\frac{x^2}{2t}}
	\E^{x\to0}\Big[\text \rme^{-2\int_0^t\diffd s\,u(B_s,t-s)}\Big]
\end{align*}
where in the last expression $B$ is a Brownian bridge from $x$ to 0. We write $B_s=x(1-\frac{s}t)-\tilde B_{t-s}$, so that $\tilde B$ is a Brownian bridge from 0 to 0, we make the change of variable $\tilde s=t-s$ and we drop the tildas:
$$
\partial_xu(x,t)=-\frac1{\sqrt{2\pi t}}\text \rme^{t-\frac{x^2}{2t}}
\E^{0\to0}\Big[\text \rme^{-2\int_0^t\diffd s\,u(x\frac s t -B_s,s)}\Big].
$$
Then, by setting $x=\sqrt2\,\rho t+z$ and integrating over $z>0$, one gets
$$
u(\sqrt2\,\rho t,t)=\frac{\text \rme^{(1-\rho^2)t}}{\sqrt{2\pi
t}}
\int_0^\infty\diffd z\, \text \rme^{-\sqrt2\,\rho z-\frac{z^2}{2t}}
\E^{0\to0}\Big[\text \rme^{-2\int_0^t\diffd s\,u(z\frac s t +\sqrt2\,\rho s-B_s,s)}\Big].
$$
For $\rho>1$, the quantity $u(z\frac s t +\sqrt2\,\rho s-B_s,s)$ goes exponentially fast to 0 as $s\to\infty$, (unless $B$ has wild fluctuations, but these events have a vanishingly small probability). Then, using the fact that $B_s$ (the value at time $s$ of a Brownian bridge over a time $t$) looks, as $t\to\infty$ for fixed $s$, more and more like a Brownian motion at time $s$, it is not very difficult (and akin to what was done in the proof of Lemma~\ref{lem:computation}) to show that
$$
\lim_{t \to \infty} \E^{0\to0}\Big[\text \rme^{-2\int_0^t\diffd s\,u(z\frac s t +\sqrt2\,\rho
s-B_s,s)}\Big] = 
\E^0\Big[\text \rme^{-2\int_0^\infty\diffd s\,u(\sqrt2\,\rho
s-B_s,s)}\Big]\quad\text{for $\rho>1$},
$$
where $B$ on the right hand side is a Brownian motion. In fact, the convergence also holds for $\rho=1$, as one can check that the quantities on either side are then equal to zero. Then, by dominated convergence,
\begin{align*}
&\int_0^\infty\diffd z\, \text \rme^{-\sqrt2\,\rho z-\frac{z^2}{2t}}
\E^{0\to0}\Big[\text \rme^{-2\int_0^t\diffd s\,u(z\frac s t +\sqrt2\,\rho s-B_s,s)}\Big]
\\  & \qquad \qquad \qquad \qquad   \xrightarrow[t\to\infty]{}
\frac1{\sqrt2\,\rho}
\E^0\Big[\text \rme^{-2\int_0^\infty\diffd s\,u(\sqrt2\,\rho
s-B_s,s)}\Big]\quad\text{for $\rho\ge1$}
\end{align*}
and
$$C(\rho)=\frac1{\sqrt{4\pi}}
\E^0\Big[\text \rme^{-2\int_0^\infty\diffd s\,u(\sqrt2\,\rho
s-B_s,s)}\Big].$$

Observe that in the point process \eqref{eqn:tildecalD} the probability that there are no particles on the right of $0$ is then 
\begin{equation*}
  \mathbf P\big(\tilde {\mathcal D}^\rho((0,\infty))=0\big)=\mathbf E\Big[\prod_k
[1-u(\sqrt2\,\rho \sigma_k-B_{\sigma_k},\sigma_k)]\Big]
=\E\Big[\text \rme^{-2\int_0^\infty\diffd s\, u(\sqrt2\,\rho s-B_{s},s)}\Big]
\end{equation*}
and therefore $C(\rho)=\frac1{\sqrt{4\pi}}   \mathbf P\big(\tilde {\mathcal D}^\rho((0,\infty))=0\big)$, as claimed.

\section{Application to branching Ornstein-Uhlenbeck processes}
\label{sec:bou}

As an application of  Theorem \ref{prop_continuity_decoration_fctn}, we study the asymptotic behavior, as $t \to \infty$, of a branching Ornstein-Uhlenbeck process with a pulling parameter that decay to $0$ as $t \to \infty$. The main motivation to study this process is the article of Cortines and Mallein \cite{CoM18}, in which it is conjectured that such a process, when undergoing selection, should exhibit unusual behaviour. In particular, the genealogy of these processes could be given by  Beta coalescents, a family that interpolates between  the Kingman and Bolthausen-Sznitman coalescents. Let us begin by introducing the branching Ornstein-Uhlenbeck process.

An Ornstein-Uhlenbeck process $X$ with spring constant $\mu$ is the solution of the stochastic differential equation
\begin{equation}
  \label{eqn:defineOU}
  \mathrm{d}X^\mu_s = - \mu X^\mu_s \mathrm{d}s + \mathrm{d}B_s,
\end{equation}
where $B$ is a Brownian motion.  It is well-known that Ornstein-Uhlenbeck processes may be represented, if $\mu > 0$, as a space-time scaled Brownian motion: given $W$ a standard Brownian motion, the process defined by
\begin{equation}
  \label{eqn:linkbmandou}
  \forall s \geq 0, \enskip X^\mu_s = X_0 \rme^{-\mu s} + \frac{\rme^{-\mu s}}{\sqrt{2\mu}}  W_{\rme^{2\mu s} - 1},
\end{equation}
is an Ornstein-Uhlenbeck with spring constant $\mu$ and initial condition $X_0$. Equation \eqref{eqn:linkbmandou} shows that, if $\mu > 0$,  the law of $X_s$, conditionally on $\{X_0=x\}$, is $\mathcal{N}(x \rme^{-\mu s}, \frac{1-\rme^{-2\mu s}}{2\mu})$. In particular, $X_s$ is then  strongly recurrent and its invariant measure is $\mathcal{N}(0, \frac{1}{2\mu})$. 

In a branching Ornstein-Uhlenbeck, since the genealogical structure of the process is independent of the motion of the particles, we continue to denote by $\mathcal{N}_t$ the set of particles alive in a branching Ornstein-Uhlenbeck process with spring constant~$\mu$ and we write $(X^\mu_s(u), u \in \mathcal{N}_s)$ for the positions of such particles. It will be convenient to work with a normalized version $\widehat{X}^\mu_s(u)$ of $X^\mu_s(u)$ that has variance $t$ so that things happen on the same scale as for the branching Brownian motion. 
This can be easily obtained by setting
\begin{equation}
  \label{def hat}
  \hat X^\mu_s(u)= \sqrt{\frac{2\mu s }{1-\rme^{-2\mu s}}} \, X^\mu_s(u).
\end{equation}
With this notation, we define the extremal point process: 
\begin{equation}
  \label{eqn:defExtremal}
  \mathcal{E}^\mu_{t} = \sum_{u \in \mathcal{N}_t} \delta_{\widehat{X}^\mu_t(u)  - \sqrt{2} t + \frac{1}{2\sqrt{2}} \log t}.
\end{equation}
Note that here the logarithmic correction is $\frac{1}{2\sqrt{2}}$ instead of $\frac{3}{2\sqrt{2}}$ as in the branching Brownian motion case ($\mu = 0$, see \eqref{BBM}). The aim of this section is to study the asymptotic behaviour of $\mathcal{E}^\mu_t$ as $\mu \to 0$ and $t \to \infty$ simultaneously.

Throughout this section, we will choose the spring constant $\mu$ as depending on the time-horizon $t$ at which we observe the positions of particles, in the sense that $\mu=\mu_t$ is kept fixed for the evolution of the branching process at all times $s\in [0,t]$. For reasons that will become clear later on, one should choose $\mu_t$ such that $\mu_t t \to \gamma\in(0,\infty]$ as $t\to \infty,$ which trivially covers the standard case where $\mu$ is fixed for all $t$'s.

The particular case $\mu_t =\gamma/t$ for some $\gamma \in (0,\infty)$ is a direct application of the results of Bovier and Hartung \cite{BoH15}. Hence we start by recalling their result.


\subsection{Extremal processes of variable speed branching Brownian motions and of branching Ornstein-Uhlenbeck processes}
\label{subsec:bbm}


For each $\sigma_b \in [0,1)$ and $\sigma_e>1$ let  $\mathcal{E}^{\sigma_b,\sigma_e}_\infty $
be a decorated Poisson point process defined as
\begin{equation}\label{def E gamma}
\mathcal{E}^{\sigma_b,\sigma_e}_\infty : = \text{DPPP} ({\sqrt{2}} C(\sigma_e)W_\infty^{\sqrt 2 \sigma_b}\rme^{-\sqrt 2 x} \diffd x , {\sigma_e}   \frakD^{\sigma_e}) ,
\end{equation}
with the parameters of the process being described as follows: Let $(X_t(u), u \in \mathcal{N}_t)$ be a branching Brownian motion and $M_t$ its maximal displacement at time $t$. Then,
\begin{itemize}
\item[-] $W_\infty^{\beta}$ is the limit of the {\it additive martingale}, previously defined in \eqref{eqn:defMartingale}. As  $(W^\beta_t, t \geq 0)$ is a non-negative martingale, it converges a.s.\@ to a limit $W^\beta_\infty$. Moreover, it is well known that a.s.\@ $W^\beta_\infty > 0$ if, and only if, $\beta \in (-\sqrt{2}, \sqrt{2})$. 

\item[-] The function $C$ is the one defined in \eqref{eqn:defineC}.
\item[-] The family of laws $(\frakD^{\rho}, \rho \geq 1)$ is the family of point processes introduced in \eqref{eqn:defineDrho}, and $c  \frakD^\rho$  is the image measure of $\frakD^\rho$ by the application $\mathcal{D} \mapsto \sum_{d_j \in \mathcal{D}} \delta_{c d_j}$, scaling the positions of the atoms by a factor $c$.
\end{itemize}

Let us now introduce the \emph{variable speed branching Brownian motion}. Let $A : [0,1] \to [0,1]$ be a twice differentiable increasing function with $A(0)=0$ and $A(1)=1$. Then, the variable speed branching Brownian motion with variance profile $A$ and time horizon $t$ is defined in the same way as a branching Brownian motion, except that particles move as Brownian motions with time-dependent variance $\sigma^2_t(s)  = A'(s/t)$ where $s \in[0,t]$ is the time of the process. In particular, the position of a particle at time $s$ is  a Gaussian random variable with variance $tA(s/t)$. 

The main result in \cite{BoH15} is the following:
\begin{thm}[Bovier and Hartung \cite{BoH15} Theorem~1.2]
\label{thm:boh}
Assume that the twice differentiable increasing function $A :\, [0,1] \to [0,1]$ satisfies
\begin{enumerate}
  \item $A(0)=0$, $A(1)=1$ and $A(x) < x$ for all $x \in (0,1)$ ;
  \item $\sigma^2_b := A'(0) < 1$ and $\sigma^2_e := A'(1)> 1$.
  \end{enumerate}
Let $(Y_s (u) ; \, u \in \mathcal{N}_s ; \, s \in [0,t] )$ denote the variable speed branching Brownian motion with variance profile~$A$ and 
\[
  \bar{\cE}^A_t = \sum_{u \in \mathcal{N}_t} \delta_{Y_t(u) - \sqrt{2}t + \tfrac{1}{2\sqrt{2}}\log t}
\] 
be its extremal point measure at time $t$. Then
\begin{enumerate}
  \item[(i)] the extremal process $\bar{\cE}^A_t$ converges in law for the topology of the vague convergence to $\mathcal{E}^{\sigma_b,\sigma_e}_\infty $.

  \item[(ii)] the maximal displacement of the process converges in law,  and for all $x \in \R$,
  \[
    \lim_{t \to \infty} \P\left( \max \bar{\cE}^A_t \leq x \right) = \P\left( \max \mathcal{E}^{\sigma_b,\sigma_e}_\infty  \leq x \right).
  \]
\end{enumerate}
\end{thm}


This Theorem is the basis for obtaining a similar result for  branching Ornstein-Uhlenbeck processes. More precisely, we will see that the case $\mu_t=\gamma/t$ is a direct consequence and that more generally the case $t\mu_t\to \gamma$ as $t\to \infty$ cane be deduced through comparison arguments.
For each $\gamma>0$, we define two constants, $c_\gamma$ and $d_\gamma$ by 
\begin{equation}
  \label{eqn:defcgammaanddgamma}
  c_\gamma := \sqrt{\frac{2\gamma}{\rme^{2\gamma}-1}} \quad \text{and}\quad d_\gamma :=\sqrt{\frac{2\gamma}{1-\rme^{-2\gamma}}}.
\end{equation}
Now for $\gamma >0$ let 
\begin{equation}
  \label{eqn:defnouveauEgamma}
  \mathcal{E}^\gamma_\infty := \mathcal{E}^{c_\gamma,d_\gamma}_\infty.
\end{equation}

In the $\gamma = \infty$ case, we set $ c_\infty=0$ and $d_\infty = \infty$, thus $W_\infty^{\sqrt{2}c_\infty} = W_\infty^{0}$ is an exponential random variable with mean $1$, the limit of the martingale associated to the Yule process $(\#\mathcal{N}_t, t \geq 0)$. 
As is shown in \cite{BoH15} (see also Theorem \ref{prop_continuity_decoration_fctn}),  $C(d_\infty) = C(\infty) =\frac{1}{\sqrt{4\pi}}$ and a point measure drawn from $\mathfrak{D}^{d_\infty} = \mathfrak{D}^{\infty}$ is a.s.\@ $\delta_0$.

We prove the following result in the rest of the section.
\begin{theorem}
\label{thm:main}
Assume that $\lim_{t \to \infty} t \mu_t = \gamma \in (0,\infty]$, then, with the above notations, we have that
\[
  \lim_{t \to \infty} \left(\mathcal{E}^{\mu_t}_{t}, \max \mathcal{E}^{\mu_t}_t\right) = \left(\mathcal{E}^\gamma_\infty, \max \mathcal{E}^\gamma_\infty \right) \quad \text{ jointly in law,}
\]
where the convergence of the point process is in the sense of the topology of vague convergence.
\end{theorem}

\begin{remark}
We prove in the forthcoming Lemma~\ref{lem:identificationLimit} that the convergence in law of a random point measure (for the topology of vague convergence) jointly with that of its maximum is equivalent to the convergence in law of $\crochet{\mathcal{P}_t,\phi}$ to $\crochet{\mathcal{P},\phi}$ for all  continuous functions~$\phi$ with support bounded from the left. This notion of convergence forms a thinner topology on the space of point measures.
\end{remark}

\begin{remark}
In the simplest case where $\mu_t =\mu $ is a constant, the theorem with \eqref{def hat} and \eqref{eqn:defExtremal} implies the following behaviour for the non-normalised positions $X_t^\mu (u)$: the position of the rightmost particle is almost surely given by
\[
\max_{u \in \mathcal N_t} X^\mu_t(u) =\sqrt{\frac{t}{\mu}} - \frac{\log t }{4\sqrt{\mu  t}}+O(t^{-1/2}), 
\] 
and the next particles are at distance of order $t^{-1/2}$ from the rightmost.
\end{remark}

We shall call the case $t \mu_t \to \infty $ the \emph{uncorrelated case}, because the extremal particles have the same distribution as the extremal particles of an i.i.d.\@ sample of Gaussian random variables. Indeed, in this regime, the dilation factor $\sqrt{2\mu_t t /(1-\rme^{-2\mu_t t})}$ diverges as $t \to \infty$, which prevents the existence of local correlations (decorations) in the limiting picture.

\subsection{The \texorpdfstring{$\mu_t=\gamma/t$}{exact connection to OU} case}

We start with the proof in the case  $\mu_t=\gamma/t$, since it is a direct application of Theorem \ref{thm:boh}.

\begin{proof}[Proof of Theorem~\ref{thm:main} in the $\mu_t=\gamma/t$ case]
Recall from \eqref{eqn:linkbmandou} that, an Ornstein-Uhlenbeck $X^{\mu_t}_s$ at time $s$ with spring-constant $\mu_t=\gamma/t$ started from 0 can be written as 
\[
X^{\mu_t}_s= \frac{\rme^{-\gamma s/t}}{\sqrt{2\gamma/t}} W_{\rme^{2\gamma s/t} -1}.
\]

For any $u \in \mathcal{N}_t$, we define 
$\big(Y_s(u),  \, s \in [0,t] \big)$ by 
\begin{equation}\label{eq_definition_variable_speed_OU}
  Y_s(u) = \sqrt{\frac{2\gamma}{\rme^{2\gamma}-1}} \rme^{\gamma s/t} X^{\gamma/t}_s(u).
\end{equation}
Clearly, $Y_s(u)$ has variance $t \frac{\rme^{2\gamma s/t}-1}{\rme^{2\gamma}-1}$. It is easily checked that the whole process $(Y_s(u), s \leq t)_{u \in \calN_t}$ is then a variable speed branching Brownian motion, with variance profile $A(x):=\frac{\rme^{2\gamma x}-1}{\rme^{2 \gamma}-1}$, which is a function satisfying the assumptions of Theorem~\ref{thm:boh} with
\[
\sigma_b^2 = A'(0) = \frac{2\gamma}{\rme^{2\gamma}-1} = c_\gamma^2,
\quad  \sigma_e^2 = A'(1)= \frac{2\gamma}{1-\rme^{-2\gamma}} = d_\gamma^2.
\]

Therefore, the extremal point process $\sum_{u \in \mathcal{N}_t} \delta_{Y_t(u)- \sqrt{2}t + \frac{1}{2\sqrt{2}}\log t}$ converges in distribution as $t\to \infty$ to a
\[\mathrm{DPPP}( {\sqrt{2}} C(\sigma_e) W_\infty^{\sqrt{2}\sigma_b}\rme^{-\sqrt{2}x}\rmd x, \,  \sigma_e \mathfrak{D}^{\sigma_e} ),\]
and the maximal atom converges as well.  Since $Y_t(u)=\hat X^{\gamma/t}_t(u)$ by \eqref{def hat}, and using the forthcoming Lemma~\ref{lem:identificationLimit}, we conclude in the joint convergence $(\cE^{\gamma/t}_t, \max \cE^{\gamma/t}_t)$ toward $(\cE^\gamma_\infty,\max \cE^\gamma_\infty)$ in law, completing the proof of Theorem~\ref{thm:main} when $\mu_t = \gamma/t$. 
\end{proof}


\subsection{Comparison of extremal processes of branching Ornstein-Uhlenbeck processes with different spring constants}
\label{subsec:extremes}

We use here Slepian-type computations to compare the extremal measures of Gaussian processes with different correlation structures. We begin with a general result on the joint convergence of point measures and their largest atom.

\begin{lemma}
\label{lem:identificationLimit}
Let $(\mathcal{P}_t, \mathcal{P}_\infty)$ be point processes on $\R$ such
that $\mathcal{P}_\infty((0,\infty))<\infty$ a.s.\@ The four following
statements are equivalent: as $t\to\infty$, 
\begin{enumerate}
  \item[(i)] $ (\mathcal{P}_t, \max \mathcal{P}_t) \rightarrow_d
(\mathcal{P}_\infty,\max \mathcal{P}_\infty)$ jointly;
  \item[(ii)] $\mathcal{P}_t \rightarrow_d \mathcal{P}_\infty$
and $\max \mathcal{P}_t \to_d \max \mathcal{P}_\infty$;
  \item[(iii)] 
	$ \E\big(\rme^{ - \crochet{\mathcal{P}_t, \phi} }\big) 
        \rightarrow   \E\big(\rme^{ - \crochet{\mathcal{P}_\infty,\phi}}\big)$
 for all continuous function $\phi$ with support bounded from the left.
\item[(iv)] $ \E\big(\rme^{ - \crochet{\mathcal{P}_t, \phi} }\big) 
        \rightarrow   \E\big(\rme^{ - \crochet{\mathcal{P}_\infty,\phi}}\big)$ for all $\mathcal{C}^\infty$ non-decreasing function $\phi$ with support bounded from the left and such that for some $a\in\R$, $\phi(x)$ is constant for $x>a$.
\end{enumerate}
\end{lemma}

The proof of this lemma being rather classical and straightforward, we postpone it to the appendix. A consequence of the above lemma is that to prove Theorem \ref{thm:main}, it is enough to prove the convergence in distribution of random variables of the form $\crochet{\mathcal{E}^{\mu_t},\phi}$, where $\phi$ is a generic non-decreasing bounded function with support bounded from the left.

We now recall that Kahane's theorem is a more general version of Slepian's lemma that allows to compare Gaussian processes with different variances. We refer to \cite[Chapter 3.1]{Bov} for a self-contained proof of Kahane's Theorem. 
\begin{thm}[Kahane's Theorem~\cite{Kah}]
\label{thm:kah}
Let $(X_j, j \leq n)$, $(Y_j, j \leq n)$ be two centred Gaussian vectors. 
Let $F$ be a twice differentiable function on $\R^n$ with bounded second derivatives, that satisfies 
\begin{align*}
  &\frac{\partial^2 F}{\partial x_i \partial x_j}(x) \geq 0 \quad \text{if}  \quad \E(X_i X_j)> \E(Y_iY_j)\\
   \text{ and } \quad&
  \frac{\partial^2 F}{\partial x_i \partial x_j}(x) \leq 0 \quad \text{if} \quad \E(X_i X_j) < \E(Y_iY_j).
\end{align*}
Then we have $\E(F(X)) \geq \E(F(Y))$.
\end{thm}

From  Kahane's Theorem~\ref{thm:kah}, we obtain Lemma~\ref{lem:extremal} below, which is useful when comparing the Laplace transform of the extremal point measures of branching Ornstein-Uhlenbeck processes with different spring constants. 

\begin{lemma}
\label{lem:extremal}
Let $\phi : \R \to \R$ be a continuous non-negative non-decreasing function. Then, for all $\mu \leq \nu \leq \infty$ and $t > 0$, we have
\[
  \E\left( \exp\left( - \crochet{\phi,\cE^\mu_t} \right) \right) \geq \E\left( \exp\left( -\crochet{\phi,\cE^\nu_t} \right) \right) ,
\]
where $\cE^\mu_t$ and $\cE^\nu_t$ are the normalized, centred extremal point measures of branching Ornstein-Uhlenbeck processes as defined in \eqref{eqn:defExtremal} when $\nu<\infty$, and $\cE^\infty_t$ is the point measure defined as
\[
  \cE^\infty_t = \sum_{u \in \mathcal{N}_t} \delta_{\hat{X}^\infty_t(u) - \sqrt{2} t + \frac{1}{2\sqrt{2}} \log t},
\]
where $(\hat{X}^\infty_t(u), u \in \mathcal{N}_t)$ is a family of i.i.d.\@ centred Gaussian random variables with variance $t$.
\end{lemma}

\begin{remark}
\label{rem:mutoinfty}
Note that as the spring constant $\mu$ increases toward $\infty$, the vector of normalized leaves $(\hat{X}^\mu_t(u), u \in \mathcal{N}_t)$ converges in law toward i.i.d.\@ Gaussian random variables with variance $t$. This can be checked by computing the covariance function of this vector, conditionally on $\mathcal{N}_t$. Therefore, we have
$
  \lim_{\mu \to \infty} \cE^\mu_t = \cE^\infty_t 
$ in law, for the topology of weak convergence,
justifying the notation.
\end{remark}

\begin{proof}
Remember that, since the branching events are independent of the spatial displacements, one can construct a branching Ornstein-Uhlenbeck with spring constant $\mu$  by first drawing its genealogical Yule tree $(\mathcal{N}_s, s\ge 0)$ then, conditionally on $(\mathcal{N}_s, s\ge 0)$ the spatial positions  $(X^\mu_s(u), u\in \mathcal{N}_s, s\ge 0)$. Thus, given two spring constants $\mu,\nu$, we can construct the two branching Ornstein-Uhlenbeck processes $X^\mu$ and $X^\nu$ using the same $(\mathcal{N}_s, s\ge 0)$. In the rest of the proof we work conditionally on $(\mathcal{N}_s, s\ge 0)$ to study the extremal processes.

For $u,v \in \mathcal{N}_t$, we denote by $\tau_{u,v}$ the time of the most recent common ancestor of $u$ and $v$. The covariance matrix of the Gaussian vectors $X^\mu$ is given by
\[
  \Var(X^\mu_t(u)) = \frac{1-\rme^{-2\mu t}}{2\mu} \text{ and }   \Cov(X^\mu_t(u),X^\mu_t(v)) = \rme^{-2\mu(t-\tau_{u,v})} \frac{1-\rme^{-2\mu \tau_{u,v}}}{2\mu}.
\]
Recall that we normalize positions to have variance $t$, setting as in \eqref{def hat}
\[
  \hat{X}^\mu_t(u) = X^\mu_t(u) \sqrt{ \frac{2\mu t}{1-\rme^{-2\mu t}} }  .
\]
As a result, we have that
\begin{equation}
  \label{eqn:cov}
 \Cov(\hat{X}^\mu_t(u),\hat{X}^\mu_t(v)) = t \frac{\rme^{2 \mu \tau_{u,v}}-1}{\rme^{2 \mu t} - 1}.
\end{equation}

Observe that when $\mu \leq \nu$ (including the case $\nu = \infty$), we have that  
\[  \Cov(\hat{X}^\mu_t(u),\hat{X}^\mu_t(v)) \geq \Cov(\hat{X}^\nu_t(u),\hat{X}^\nu_t(v)),\]
for all $u,v \in \mathcal{N}_t$. Indeed, it is easy to verify that for all $0<s < t$ fixed the function $\mu \mapsto \frac{\rme^{2 \mu s} - 1}{\rme^{2 \mu t} - 1}$ is non-increasing in $\mu  \in \R$.

We start by showing the result for  $\phi : \R \to \R$, a smooth non-negative non-decreasing function, such that $\phi'$ has compact support. Then the function
\[
  F : x \in \R^{\mathcal{N}_t} \mapsto \exp\left( -\sum_{u \in \mathcal{N}_t} \phi(x_u) \right),
\]
is twice differentiable and constant outside of a compact, hence its second derivatives are bounded. It satisfies
\[
  \frac{\partial^2 F}{\partial x_i \partial x_j}(x) = \phi'(x_i) \phi'(x_j) \exp\left( - \sum_{u \in \mathcal{N}_t} \phi(x_u)\right) \geq 0,
  \quad \text{for all $i \neq j \in \mathcal{N}_t$,}
\]
by monotonicity of $\phi$. Thus, we can apply Kahane's Theorem \ref{thm:kah}, and we have that for all $\mu \leq \nu \leq \infty$, 
\[\E\left( F(\hat {X}_t^\mu) \middle| \mathcal{N}_t \right) \geq \E\left( F(\hat{X}_t^\nu) \middle| \mathcal{N}_t\right).\]
Therefore, averaging over the genealogical tree $(\mathcal{N}_t, t \geq 0)$, we obtain that
\[
  \mu \in (-\infty, \infty] \mapsto  \E\left( \exp\left( - \crochet{\cE_t^\mu,\phi} \right) \right)
\]
is non-increasing.

To conclude, note that any continuous non-decreasing non-negative function $\phi$ can be approached from below by a sequence $(\phi_n, n \geq 1)$ of smooth non-decreasing functions with derivatives having compact support. Moreover,
\[
  \lim_{n \to \infty} \E\left( \exp\left(-\crochet{\cE_t^\mu, \phi_n}\right) \right) = \E\left( \exp\left(-\crochet{\cE_t^\mu,\phi}\right) \right) 
\]
by monotone convergence. Hence, we conclude that $\mu \mapsto \E\left( \exp\left( - \crochet{\cE_t^\mu,\phi} \right) \right)$ is non-increasing.
\end{proof}

\subsection{Proof of Theorem \ref{thm:main}}
\label{subsec:terminal}

We complete the proof of Theorem \ref{thm:main} in this section. We start with the observation that the family of limiting point measures $(\cE^\gamma_\infty, \gamma \in (0,\infty])$, defined in \eqref{eqn:defnouveauEgamma}, is continuous in distribution.
\begin{proposition}
\label{prop:continuity}
The family $\big( (\cE^\gamma_\infty, \max \cE^\gamma_\infty); \gamma \in (0, \infty] \big)$ is continuous in law. Otherwise said, as per Lemma~\ref{lem:identificationLimit}, for all continuous $\phi : \R \to \R_+$ non-decreasing with bounded support from the left, the function
\[
  \gamma \in (0,\infty] \mapsto  \E\left( \rme^{ - \crochet{\cE^\gamma_\infty,\phi} } \right) \quad \text{ is continuous.}
\]
\end{proposition}

\begin{proof}
 Let $\phi$ be a continuous non-decreasing function, with support bounded from the left. For any $\gamma > 0$, by Campbell's formula, we have
\begin{equation}\label{eq:conv_Egamma_Laplace}
  \E\left( \rme^{ - \crochet{\cE^\gamma_\infty,\phi} } \right) =
  \E\left(\exp\left( - \int_\R \E\left( 1 - \rme^{-\crochet{\calD^{d_\gamma},\phi( d_\gamma\cdot \, + z)}} \right) \sqrt{2}{C}(d_\gamma) W_\infty^{\sqrt 2c_\gamma} \rme^{-\sqrt{2}z}\dd z \right) \right).
\end{equation}
We observe that ${C}(d_\gamma), W_\infty^{\sqrt 2 c_\gamma}$ as well as $\E\Big( 1 - \rme^{-\crochet{\calD^{d_\gamma},\phi( d_\gamma\cdot \, + z)}} \Big)$ are non-negative for all $\gamma > 0$ and hence the exponential term on the right-hand side of \eqref{eq:conv_Egamma_Laplace} is bounded by $1$. Therefore, by dominated convergence, it is enough to prove that each of the above functions is continuous.

It is obvious from the definition that both functions $\gamma \mapsto c_\gamma$ and $\gamma \mapsto d_\gamma$ are continuous in $\gamma$ with $c_\gamma \in (0,1)$ and $d_\gamma>1$ for all $\gamma > 0$. At the same time, Theorem~\ref{prop_continuity_decoration_fctn} says that both 
\[
\gamma \mapsto {C}(\rho)
\quad \text{and} \quad
\rho \mapsto \E\left( 1 - \rme^{-\crochet{\calD^\rho,\phi(\rho \cdot + z)}} \right)
\]
are continuous in $\rho > 1$, by dominated convergence. Finally, Biggins \cite{Big92} proved that the convergence of the additive martingale $W^{\sqrt{2}\rho}$ is uniform on compact subsets of $(-1,1)$, {\it i.e.}\@ for all $\epsilon \in (0,1)$, we have
\[
  \lim_{t \to \infty} \sup_{\rho \in [\epsilon-1,1-\epsilon]} \left| W_t^{\sqrt 2\rho} - W^{\sqrt 2 \rho}_\infty \right| = 0 \quad \text{a.s.}
\]
As a result, we deduce that $W^{\sqrt 2 \rho}_\infty$ is continuous in $\rho$, completing the proof.
\end{proof}

We now show that the point process $\cE^\infty_t$ defined in Lemma~\ref{lem:extremal} converges in law, as $t \to \infty$, to the Poisson point process $\cE^\infty_\infty$ defined in \eqref{eqn:defnouveauEgamma}, jointly with its maximum. Recall that
\[
  \cE^\infty_t = \sum_{u \in \mathcal{N}_t} \delta_{\hat{X}^\infty_t(u) - \sqrt{2} t + \frac{1}{2\sqrt{2}} \log t} \quad \text{and} \quad \cE^\infty_\infty \sim PPP\left( \tfrac{1}{\sqrt{2\pi}} W^0_\infty \rme^{-\sqrt{2} x} \dd x \right),
\]
where $(\hat{X}^\infty_t(u), u \in \mathcal{N}_t)$ are i.i.d.\@ centred Gaussian random variables with variance $t$.
\begin{lemma}
\label{lem:remark}
We have
\[
  \lim_{t \to \infty} \left( \cE^\infty_t, \max \cE^\infty_t \right) = \left( \cE^\infty_\infty, \max \cE^\infty_\infty \right) \quad \text{ in law}.
\]
\end{lemma}

\begin{proof}
Note this result can be straightforwardly deduced from standard extreme values theory for Gaussian processes. We include a direct self-contained proof which furthermore demonstrates how our toolbox can be used. Recall from Lemma~\ref{lem:identificationLimit} that to prove the joint convergence of $\cE^\infty_t$ and its maximum, it is enough to prove the convergence of $\E\left( \exp\left(- \crochet{\cE^\infty_t,f}\right)\right)$ for all  non-decreasing continuous functions $f$ with support bounded from the left.

Observe, by Campbell's formula for Poisson point processes, that
\[
  \E\left( \exp\left(- \crochet{\cE^\infty_\infty, f} \right) \middle| W^0_\infty \right) =  \exp\left( -W^0_\infty \int \left(1-\rme^{-f(y)}\right) \frac{\sqrt{2}\rme^{-\sqrt{2} y}}{\sqrt{4\pi}} \dd y \right).
\]
Therefore, as $W^0_\infty$ is distributed as a standard exponential random variable, we have
\begin{equation}
  \label{eqq2}
  \E\left( \exp\left(- \crochet{\cE^\infty_\infty, f} \right)\right) = \left(1 + \frac{1}{\sqrt{2\pi}}\int \left(1 - \rme^{-f(y)}\right) \rme^{-\sqrt{2} y} \dd y \right)^{-1}.
\end{equation}

On the other hand, conditioning with respect to $\#\mathcal{N}_t$ the number of leaves at time $t$, and writing $X_t$ for a Gaussian random variable with variance $t$ and $m_t = \sqrt{2}t - \frac{1}{2\sqrt{2}} \log t$, we have
\[
  \E\left( \exp\left(- \crochet{\cE^\infty_t, f} \right) \right) = \E\left(\E\left( \rme^{-f\left(X_t- m_t\right)} \right)^{\#\mathcal{N}_t} \right).
\]
As $\#\mathcal{N}_t$ is a geometric random variable with parameter $\rme^{-t}$, we have
\begin{align}
\E\left( \exp\left(- \crochet{\cE^\infty_t, f} \right) \right) 
  &= \frac{\rme^{-t} \E\left( \rme^{-f\left(X_t- m_t\right)} \right)}{1 - (1-\rme^{-t})\E\left( \rme^{-f\left(X_t-m_t\right)} \right)}\nonumber \\
  &= \frac{ \E\left( \rme^{-f\left(X_t - m_t\right)} \right)}{\rme^t\E\left( 1 -\rme^{-f\left(X_t- m_t\right)} \right) + \E\left( \rme^{-f\left(X_t- m_t\right)} \right)}\label{eqq1}
\end{align}
Therefore, to complete the proof, it is enough to prove that
\begin{equation}
  \label{matp}
  \lim_{t \to \infty} \rme^t\E\left( 1 -\rme^{-f\left(X_t- m_t\right)} \right) = \frac{1}{\sqrt{2\pi}}\int (1 - \rme^{-f(y)}) \rme^{-\sqrt{2} y} \dd y,
\end{equation}
which implies that \eqref{eqq1} converges to \eqref{eqq2} as $t \to \infty$.

We now turn to the proof of \eqref{matp}. By change of variables, we have
\begin{align*}
\E\left( 1-\rme^{-f\left(X_t- m_t\right)} \right)
&= \int  (1-\rme^{-f\left(x- m_t\right)}) \frac{\rme^{-x^2/2t}}{\sqrt{2\pi t}} \dd x\\
&=  \int \frac{(1-\rme^{-f(y)})}{\sqrt{2\pi t}} \rme^{-\sqrt{2}y} \rme^{ - t + \frac{1}{2} \log t} \rme^{- \frac{y^2}{2t}+y \frac{\sqrt{2} \log t}{4t} - \frac{(\log t)^2}{16t}}  \dd y.
\end{align*}
Hence, as the support of $y \mapsto 1 - \rme^{-f(y)}$ is bounded from the left, we can apply the dominated convergence theorem in the above equation yielding, as $t \to \infty$,
\[
  \E\left( 1-\rme^{-f\left(X_t- m_t\right)} \right) \sim \frac{e^{-t}}{\sqrt{2\pi}}\int (1 - \rme^{-f(y)}) \rme^{-\sqrt{2} y} \dd y
\]
concluding the proof.
\end{proof}

Finally, we use the Kahane estimate to control the branching Ornstein-Uhlenbeck process with pulling strength $\mu_t$ by branching Ornstein-Uhlenbeck processes with pulling strength $(\gamma \pm \epsilon)/t$.

\begin{proof}[Proof of Theorem~\ref{thm:main}]
We denote by $(X^{\mu_t}_t(u), u \in \mathcal{N}_t)$ the positions at time $t$ of a branching Ornstein-Uhlenbeck process with spring constant $\mu_t$ (recall that the spring constant $\mu_t$ remains constant throughout the process, up to time $t$) and assume that
\[
  \lim_{t \to \infty} t \mu_t = \gamma \in (0,\infty].
\]

We first consider the case $\gamma < \infty$. Let $0< \underline{\gamma} < \gamma < \bar{\gamma}$. For $t$ large enough $ \underline{\gamma}/t < \mu_t < \bar{\gamma}/t$. Thus, by Lemma~\ref{lem:extremal}, 
\[
  \E\left( \exp\left( - \crochet{\phi,\cE^{\underline{\gamma}/t}_t} \right) \right) \leq \E\left( \exp\left( - \crochet{\phi,\cE^{\mu_t}_t} \right) \right) \leq \E\left( \exp\left( - \crochet{\phi,\cE^{\bar{\gamma}/t}_t} \right) \right),
\]
for all $\phi$ continuous non-decreasing functions $\R \to \R_+$. 

As a result, taking $t \to \infty$, and supposing furthermore that $\phi$ has bounded support on the left, combining Lemma~\ref{lem:identificationLimit} and Theorem~\ref{thm:boh}, we obtain that 
\begin{align*}
  \liminf_{t \to \infty}&\enskip\E\left( \exp\left( - \crochet{\phi,\cE^{\mu_t}_t} \right) \right) \geq \E\left( \exp\left( - \crochet{\phi,\cE_\infty^{\underline{\gamma}}} \right)\right)\\
  \limsup_{t \to \infty}&\enskip\E\left( \exp\left( - \crochet{\phi,\cE^{\mu_t}_t} \right) \right) \leq \E\left( \exp\left( - \crochet{\phi,\cE_\infty^{\bar{\gamma}}} \right)\right).
\end{align*}
Now, letting $\underline{\gamma} \uparrow \gamma$ and $\bar{\gamma} \downarrow \gamma$, using Proposition \ref{prop:continuity} we obtain 
\[
  \lim_{t \to \infty} \E\left( \exp\left( - \crochet{\phi,\cE^{\mu_t}_t} \right) \right) = \E\left( \exp\left( - \crochet{\phi,\cE^{\gamma}} \right)\right).
\]
We conclude by Lemma~\ref{lem:identificationLimit} that $(\cE^{\mu_t}_t,\max \cE^{\mu_t}_t)$ converge toward $(\cE^\gamma,\max \cE^\gamma)$.

We now consider the case $\gamma = \infty$. If $\lim_{t \to \infty} t \mu_t = \infty$, then for all $\underline{\gamma} > 0$,  one has $\mu_t \geq \underline{\gamma}/t$ for all $t$ large enough. One the other hand, $(X^{\mu_t}_t(u), u \in \mathcal{N}_t)$ is straightforwardly ``more correlated'' than i.i.d.\@ Gaussian random variables (formally corresponding to the case $\gamma=\infty$). Hence, using again Lemma~\ref{lem:extremal}, then Lemma~\ref{lem:identificationLimit} and Theorem~\ref{thm:boh} for the lower bound, and Lemma~\ref{lem:remark} for the upper bound, we obtain 
\begin{align*}
  \liminf_{t \to \infty} \E\left( \exp\left( - \crochet{\phi,\cE^{\mu_t}_t} \right) \right) &\ge \E\left( \exp\left( - \crochet{\phi,\cE^{\underline\gamma}_\infty} \right)\right),\\
   \limsup_{t \to \infty} \E\left( \exp\left( - \crochet{\phi,\cE^{\mu_t}_t} \right) \right) &\le \E\left( \exp\left( - \crochet{\phi,\cE^{\infty}_\infty} \right)\right)
\end{align*}
for all smooth increasing function $\phi:\R \to [0,1]$ such that $\phi'$ has compact support.
Letting $\underline \gamma \to \infty$ concludes the proof of Theorem~\ref{thm:main}.
\end{proof}

\section{Open questions and future work}
\label{sec:op}

The cases  $t\mu_t \to \gamma \in (0,\infty)$  interpolate between the \emph{uncorrelated case} and the branching Brownian motion regime ($\mu_t = 0$). Notice, though, that the multiplicative factor of the logarithmic correction remains equal to $\frac{1}{2\sqrt{2}} $ (as in the  uncorrelated case) and not $\frac{3}{2\sqrt{2}}$ (as in the branching Brownian motion). We believe that there is a second transition when $t\mu \to 0$ where one  gradually goes from the $\frac{1}{2\sqrt{2}} \log t $ correction to $\frac{3}{2\sqrt{2}} \log t$ while the decoration measure always is $\calD^1,$ which is the decoration of the branching  Brownian motion.

More precisely, it is predicted in \cite{DMS} that ${C}(\rho) \sim \kappa (\rho-1)$ as $\rho \to 1$, with the same constant $\kappa$ as in~\eqref{DPPP du BBM}.  Note that $\kappa \approx 1.18$ is also the constant such that $\lim_t \P(M_t \geq \sqrt{2} t - \frac{3}{2\sqrt{2}} \log t + y) \sim \frac{\kappa}{\sqrt 2} y e^{-\sqrt{2}y}$, as $y\to \infty$. This constant is proved to exist for all branching random walks in \cite[Proposition~4.1]{Aid}. Note that in \cite{DMS} the function $\Phi$ defined by 
\[
u(ct,t) \sim \frac{\rme^{-t(c^2/4-1)}}{\sqrt{4\pi t }} \Phi(c) \quad \text{ as } t \to \infty
\]
where $u$ is the solution of the Fisher-KPP equation $\partial_t u =\partial^2_x u +u(1-u)$ started from the Heavyside initial condition is the analogue of $C$. The exact correspondence between the functions $\Phi$ and $C$ is
\[
C(\rho) =\frac{\rho}{\sqrt{4\pi}} \Phi(2\rho).
\]
Our factor $\kappa$ is thus given by the constant denoted $2\alpha$ in \cite{DMS} (see Equation (73) there).

On the other hand,  we also know from  \cite{Madaule}, that for the additive martingale $W^\beta$
\[
  \lim_{\beta \to \sqrt{2}-} \frac{W^\beta_\infty}{\sqrt{2}-\beta} = \sqrt{2} Z_\infty,
\]
with $Z_\infty$ the limit of the derivative martingale. Since $d_\gamma \simeq 1+\gamma/2$ and $c_\gamma \simeq 1-\gamma/2$ when $\gamma \to 0$, we see that 
\[
 C(d_\gamma) W^{\sqrt 2 c_\gamma}_\infty \simeq  \frac{\kappa \gamma^2}{2} Z_\infty \quad \text{as $\gamma \to 0$}.
\]
Since $\gamma^2 e^{-\sqrt 2 x} = e^{-\sqrt 2 (x-\sqrt 2 \log \gamma)}$, the extremal point process $\cE^\gamma_\infty$ is roughly  $\cE_\infty$, the centred extremal point process of the standard  branching Brownian motion see \eqref{DPPP du BBM},  
shifted to the left  by $\sqrt 2 | \log \gamma|+\mathcal O(1)$ (as $\gamma\to 0$).
This  might suggest that the above-mentioned 
intermediate logarithmic corrections between $\frac{1}{2\sqrt{2}}$ and $\frac{3}{2\sqrt{2}}$ should appear for $\mu_t = t^{-\alpha}$ with $\alpha \in (1,3/2)$ , 
and the extremal point measure would be the same as for the branching Brownian motion as soon as $\mu_t = o(t^{-3/2})$.
This would complement the recent work \cite{BoH18} on a similar phenomenon for branching Brownian motion with piecewise constant variance.

It may be worth noting that our model is notably different from the one studied by Kiestler and Schmidt \cite{KS} which yields a different interpolation between the uncorrelated case and the branching Brownian motion. In that later model, the extremal model is a Poisson point process without decoration, but the logarithmic correction of the median of the maximal displacement interpolates between $-\frac{1}{2 \sqrt{2}}$ and $\frac{-3}{2\sqrt{2}}$. On the contrary, in our case, the decoration of the extremal processes interpolate continuously between the absence of decoration of the uncorrelated case and the decoration of the branching Brownian motion. However, the logarithmic correction does not interpolate continuously on the scale of parameters we are considering.

The case $\mu<0$ is also  interesting and is not covered in the present work. Notice that in the case $\mu>0$ we rely heavily on  the results from Bovier and Hartung \cite{BoH15}. However we think that the $\mu<0$ case corresponds to that of decreasing variances for the variable speed branching Brownian motion for which results concerning the position of the maximum are known (see e.g. Maillard and Zeitouni \cite{MaZ}), but not concerning the full extremal point process.

\appendix

\section{Proof of Lemma \ref{lem:identificationLimit}}
\begin{proof}
Obviously, (i) implies (ii) and (iii) implies (iv). It remains to proove
that (ii) implies (iii) and (iv) implies (i). 

We start by proving that (ii) implies (iii). First consider the case of
a non-negative continuous function $\phi$ with support bounded from the
left, and introduce for
$A\in\R$ 
$$\phi^A: x\mapsto \begin{cases}\phi(x)&\text{if $x<A$}\\
	(A+1-x)\phi(A)&\text{if $x\in[A,A+1]$}\\
	0&\text{if $x>A+1$}.\end{cases}$$
The function $\phi^A$ is continuous compactly supported, hence by (ii) we have
\[
  \lim_{t \to \infty} \E\big( \rme^{ - \crochet{\mathcal{P}_t,\phi^A}
}\big) = \E\big( \rme^{ -\crochet{\mathcal{P}_\infty,\phi^A}}  \big).
\]
By triangular inequality,
$$\begin{aligned}\Big| 
\E\big( \rme^{ - \crochet{\mathcal{P}_t,\phi} }\big) 
-
\E\big( \rme^{ - \crochet{\mathcal{P}_\infty,\phi} }\big) 
\Big|
\le&\quad\,
\Big|
\E\big( \rme^{ - \crochet{\mathcal{P}_t,\phi} }\big) 
-
\E\big( \rme^{ - \crochet{\mathcal{P}_t,\phi^A} }\big) 
\Big|
\\&
+
\Big|
\E\big( \rme^{ - \crochet{\mathcal{P}_t,\phi^A} }\big) 
-
\E\big( \rme^{ - \crochet{\mathcal{P}_\infty,\phi^A} }\big) 
\Big|
+
\Big|
\E\big( \rme^{ - \crochet{\mathcal{P}_\infty,\phi^A} }\big) 
-
\E\big( \rme^{ - \crochet{\mathcal{P}_\infty,\phi} }\big) 
\Big|.
\end{aligned}
$$
Moreover, as $\phi$ is non-negative, we have for all $t\ge0$ and also for
$t=\infty$:
\[
  \Big|  \E\big( \rme^{ - \crochet{\mathcal{P}_t,\phi} } \big)
-  \E\big( \rme^{ - \crochet{\mathcal{P}_t,\phi^A} } \big) \Big|
  \leq \P\left( \max \mathcal{P}_t \geq A \right).
\]
Hence, by convergence of $\max \mathcal{P}_t$, we have
\[
  \limsup_{t \to \infty} 
\Big|\E\big( \rme^{ - \crochet{\mathcal{P}_t,\phi} }\big) 
-
\E\big( \rme^{ - \crochet{\mathcal{P}_\infty,\phi} }\big) 
\Big|\le 2 \P\left( \max \mathcal{P}_\infty \geq A \right).
\]
As the right hand side goes to zero as $A\to\infty$, we have proved (iii)
for non-negative functions. Now consider an arbitrary continuous function
$\phi$ with support bounded on the left, and write
$$\phi = \phi_+ - \phi_-\qquad\text{where $\phi_+(x)=\max\big(\phi(x),0\big)$
and $\phi_-(x)=\max\big(-\phi(x),0\big)$}.$$
Then, for any $\alpha,\beta\ge0$, the function
$\alpha\phi_++\beta\phi_-$ is continuous non-negative  with support bounded
on the left and, therefore,
$$\lim_{t\to\infty}
\E\Big( \rme^{ - \alpha\crochet{\mathcal{P}_t,\phi_+}
-\beta\crochet{\mathcal{P}_t,\phi_-} }\Big) 
=
\E\Big( \rme^{
- \alpha\crochet{\mathcal{P}_\infty,\phi_+}-\beta\crochet{\mathcal{P}_\infty,\phi_-} }\Big).$$
We conclude that
$(\crochet{\mathcal{P}_t, \phi_+},\crochet{\mathcal{P}_t, \phi_-})$ jointly
converge in law toward $(\crochet{\mathcal{P}_\infty,
\phi_+},\crochet{\mathcal{P}_\infty,
\phi_-})$. Therefore, $\crochet{\mathcal{P}_t,\phi}$ converges as well toward
$\crochet{\mathcal{P}_\infty,\phi}$, which implies that (iii) holds.

\medskip

We now prove that (iv) implies (i).
Let $f$ be a $\mathcal{C}^\infty$ non-decreasing function such that $f(x)=0$ for $x<0$ and
$f(x)=1$ for $x>1$. For any $y\in\R$ and $\epsilon>0$, we set
$f_{\epsilon,y}(x)=f\big(\epsilon^{-1}(x-y)\big)$.

Noting that 
$f_{\epsilon,y}(x)\le\ind{x>y}\le f_{\epsilon,y-\epsilon}(x)$, we have for all $(y_1,\ldots,y_n)\in\R^n$, $(\lambda_1,\ldots,\lambda_n)\in\R_+^n$
and $\epsilon>0$:
$$\E\Big( \rme^{ - \sum_{i} \lambda_i
\crochet{\mathcal{P}_t,f_{\epsilon,y_i-\epsilon}}}\Big)
\le
\E\Big( \rme^{ - \sum_{i} \lambda_i \mathcal P_t((y_i,\infty))}
\Big)
\le
\E\Big( \rme^{ - \sum_{i} \lambda_i
\crochet{\mathcal{P}_t,f_{\epsilon,y_i}}}\Big).
$$
As $t\to\infty$, the two bounds converge by (iv) applied to the functions
$\sum_i\lambda_i f_{\epsilon,y_i}$ and
$\sum_i\lambda_i f_{\epsilon,y_i-\epsilon}$
$$\begin{aligned}
\E\Big( \rme^{ - \sum_{i} \lambda_i
\crochet{\mathcal{P}_\infty,f_{\epsilon,y_i-\epsilon}}}\Big)
\le&
\liminf_{t\to\infty}\E\Big( \rme^{ - \sum_{i} \lambda_i \mathcal P_t((y_i,\infty))}
\Big)
\\
\le&\limsup_{t\to\infty}\E\Big( \rme^{ - \sum_{i} \lambda_i \mathcal P_t((y_i,\infty))}
\Big)
\le
\E\Big( \rme^{ - \sum_{i} \lambda_i
\crochet{\mathcal{P}_\infty,f_{\epsilon,y_i}}}\Big).
\end{aligned}
$$
Note that
$f_{\epsilon,y}(x)\to\ind{x>y}$ and
$f_{\epsilon,y-\epsilon}(x)\to\ind{x\ge y}$ as $\epsilon\to0$. Hence one
gets
\begin{equation}\begin{aligned}
\E\Big( \rme^{ - \sum_{i} \lambda_i
\mathcal{P}_\infty([y_i,\infty))} \Big)
\le&
\liminf_{t\to\infty}\E\Big( \rme^{ - \sum_{i} \lambda_i \mathcal P_t((y_i,\infty))}
\Big)
\\
\le&\limsup_{t\to\infty}\E\Big( \rme^{ - \sum_{i} \lambda_i \mathcal P_t((y_i,\infty))}
\Big)
\le
\E\Big( \rme^{ - \sum_{i} \lambda_i
\mathcal{P}_\infty((y_i,\infty))}\Big).
\end{aligned}\label{liminfsup}
\end{equation}
We conclude that 
$(\mathcal P_t((y_i,\infty)), i\le n)$
jointly converge in law to
$(\mathcal P_\infty((y_i,\infty)), i\le n)$ as $t\to\infty$, except at
discontinuity points $y_i$ where $\mathcal P_\infty(\{y_i\})>0$ with
positive probability. Hence, 
$\mathcal P_t$ converges in law to $\mathcal P_\infty$ for the topology of
vague convergence.

In \eqref{liminfsup}, add one extra pair $(\lambda,y)$ to the $\lambda_i$, $y_i$, and send $\lambda$ to infinity.
Noticing that for $A\ge0$ that 
$$\E(A\ind{\max\mathcal P_t \le y}) \le
\E\big( A \rme^{ - \lambda \mathcal P_t((y,\infty))}\big)\le
\E(A\ind{\max\mathcal P_t \le y})+\rme^{-\lambda} \E(A),$$
one gets
\begin{align*}
\E\Big( \rme^{ - \sum_{i} \lambda_i
\mathcal{P}_\infty([y_i,\infty))}\ind{\max\mathcal P_\infty \le y}) \Big)
\le&
\liminf_{t\to\infty}\E\Big( \rme^{ - \sum_{i} \lambda_i \mathcal P_t((y_i,\infty))}
\ind{\max\mathcal P_t \le y})\Big)
\\
\le&\limsup_{t\to\infty}\E\Big( \rme^{ - \sum_{i} \lambda_i \mathcal P_t((y_i,\infty))}
\ind{\max\mathcal P_t \le y})\Big)\\
\le&
\E\Big( \rme^{ - \sum_{i} \lambda_i
\mathcal{P}_\infty((y_i,\infty))}\ind{\max\mathcal P_\infty \le y})\Big).
\end{align*}
Hence $(\mathcal P_t, \max\mathcal P_t)$ converges to $(\P_\infty, \max\mathcal P_\infty)$ in law jointly.
\end{proof}

\paragraph*{Acknowledgements:} 
A.C.'s work is supported by the Swiss National Science Foundation~200021\underline{{ }{ }}163170.
B.M. and É.B. are partially funded by ANR-16-CE93-0003 (ANR MALIN). B.M. is also partially funded by a PEPS JCJC 2019 grant from CNRS. J.B. is partially supported by ANR grants ANR-14-CE25-0014 (ANR GRAAL) and ANR-14-CE25-0013 (ANR NONLOCAL).

\bibliographystyle{alpha}

\end{document}